\documentclass[12pt]{amsart}
\usepackage{amssymb}
\usepackage{amsbsy}
\usepackage{amscd}

\usepackage{mathrsfs}
\usepackage{nicefrac}

\usepackage{epic}
\usepackage{eepic}
\makeatletter
\renewcommand{\thesubsection}{\thesection(\@roman\c@subsection)}
\makeatother
\usepackage{verbatim}
\usepackage{version}
\usepackage{color}
\newenvironment{NB}{
\color{red}{\bf NB}. \footnotesize
}{}

\excludeversion{NB}
\excludeversion{NB2}
\newtheorem{Theorem}[equation]{Theorem}

\newtheorem{Lemma}[equation]{Lemma}
\newtheorem{Proposition}[equation]{Proposition}

\theoremstyle{definition}

\theoremstyle{remark}
\newtheorem{Remark}[equation]{Remark}

\numberwithin{equation}{section}

\newcommand{\lemref}[1]{Lemma~\ref{#1}}

\newcommand{\defeq}{\overset{\operatorname{\scriptstyle def.}}{=}}
\newcommand{\CC}{{\mathbb C}}
\newcommand{\ZZ}{{\mathbb Z}}
\newcommand{\QQ}{{\mathbb Q}}

\newcommand{\algsl}{\operatorname{\mathfrak{sl}}} 

\newcommand{\g}{{\mathfrak g}}

\renewcommand{\MR}[1]{}

\newcommand{\wt}{\operatorname{wt}}
\newcommand{\bU}{\mathbf U}
\newcommand{\M}{\mathfrak M}
\newcommand{\La}{\mathfrak L}
\newcommand{\Mreg}{\M^{\operatorname{reg}}}
\usepackage{url}

\makeindex

\usepackage[
bookmarks=true,
colorlinks=true]{hyperref}

\begin{document}
\title[Affine cellularity of quantum affine algebras]
{Affine cellularity of quantum affine algebras
}
\author{Hiraku Nakajima}
\address{Research Institute for Mathematical Sciences,
Kyoto University, Kyoto 606-8502,
Japan}
\email{nakajima@kurims.kyoto-u.ac.jp}

\maketitle

\appendix
\setcounter{section}{1}
\section*{}

This is an appendix to Cui's paper \cite{Cui} showing that the
modified quantum affine algebra $\widetilde\bU = \widetilde\bU_q(\g)$
of level 0 (more precisely its quotients, BLN algebras) is affine
cellular in the sense of Koenig and Xi \cite{KX}. The proof is based
on the structure of cells of $\widetilde\bU$, studied previously in
\cite{MR2066942}, the author's joint work with Beck.
We here give a proof based on \cite[Lemma~6.17]{MR2066942}, together
with a property of the bilinear form introduced in
\cite{MR2074599}. Note that \cite[Lemma~6.17]{MR2066942} and the
bilinear form are crucial ingredients for the study of the structure
of cells in \cite{MR2066942}. In this sense the following proof is
more direct and fundamental than one in \cite{Cui}.

We also prove that cell ideals are idempotent, and hence
\cite[Th.~4.4]{KX} is applicable. Therefore BLN algebras are of finite
global dimension, and its derived category admits a stratification
whose sections are equivalent to derived categories of representation
rings of products of general linear groups.
The proof is, more or less, a simple
observation once we remember that $\widetilde\bU$ has the highest
weight theory.

We also give a remark, which explain why it is {\it natural\/} to
expect that $\widetilde\bU$ is affine cellular, in view of geometry of
quiver varieties, when the underlying affine Lie algebra $\g$ is
symmetric, i.e., an untwisted affine Lie algebra of type $ADE$.
However, it should be emphasized that convolution algebras and
$\widetilde\bU$ (or its quotients) are possibly different, and we do
not know whether convolution algebras are affine cellular or not
unfortunately.

The existence of the affine cellular structure on $\widetilde{\mathbf
  U}$ is a simple consequence of \cite{MR2066942,MR2074599}.
However a point is its usefulness and generality \cite{KX}, and hence
it is worthwhile to note that $\widetilde\bU$ is affine cellular. It
is this reason why we write this short note to emphasize this
observation again after \cite{Cui}, and to clarify where the affine
cellular structure come from for $\widetilde\bU$.

One of applications of the theory of affine cellular algebras is a
classification of simple modules \cite[\S3]{KX}. For $\widetilde\bU$,
it reproduces a well-known classification, namely simple modules are
parametrized by Drinfeld polynomials.

On the other hand, it is probably not previously known that BLN
algebras are of finite global dimension and their derived cateogories
admit stratification. Therefore it is really useful to point out that
the theory of affine celluar algebras is applicable to $\widetilde\bU$.

Our notation follows Cui's paper \cite{Cui} and the author's previous ones
\cite{MR2066942,MR2074599}, as well as \cite{MR2058969} for geometric
objects.

\subsection*{Acknowledgment}

The author thanks Dr.\ Ryosuke Kodera for comments on an earlier
version of this paper, e.g., the observation that $i_*$ below is
injective.

\subsection{Laurent polynomial valued bilinear form}

Let $\lambda = \sum_{i\in I_0} m_i\varpi_i$ be a level $0$ dominant
weight, and $V(\lambda)$ be the corresponding extremal weight module. 

In \cite[\S4]{MR2066942} we introduced a $\bU$-homomorphism
\begin{equation*}
    \Phi_\lambda\colon V(\lambda)\to
    \widetilde V(\lambda) \defeq \bigotimes V(\varpi_i)^{\otimes m_i},
\end{equation*}
sending $u_\lambda$ to the tensor product $\widetilde u_\lambda \defeq
\bigotimes u_{\varpi_i}^{\otimes m_i}$ of extremal weight vectors, and
then analyzed the structure of $V(\lambda)$ via $\Phi_\lambda$.
Each factor $V(\varpi_i)$ has a $\bU'$-linear automorphism $z_i$ of
weight $d_i\delta$. We introduce variables $z_{i,\mu}$ ($\mu=1,\dots,
m_i$) as the automorphism for the $\mu^{\mathrm{th}}$-factor
$V(\varpi_i)^{\otimes m_i}$, and regard them as automorphisms of
$\widetilde V(\lambda)$.

Let $(\ ,\ )$ be the bilinear form on $V(\lambda)$ introduced in
\cite[\S4]{MR2074599}.
Recall (\cite[\S4]{MR2074599}) that we have defined
$\QQ(q_s)[z_{i,\nu}^\pm]_{i\in I_0,\nu=1,\dots,m_i}$-valued bilinear
form $(\!(\ ,\ )\!)$ on $\widetilde V(\lambda)$, which is related to
$(\ ,\ )$ on $V(\lambda)$ via $\Phi_\lambda$ by
\begin{equation}\label{eq:2}
    (u,v) = \left[(\!(\Phi_\lambda(u), \Phi_\lambda(v) )\!)
      \prod_{i} \frac1{m_i!} 
      \prod_{\nu\neq\mu}\left( 1 - z_{i,\mu} z_{i,\nu}^{-1}\right)
    \right]_1.
\end{equation}
Here $[\ ]_1$ denotes the constant term.
From its definition, we also have
\begin{equation}\label{eq:3}
    (\!(f(z) u, v)\!) = f(z) (\!(u,v)\!), \quad
    (\!(u, g(z) v)\!) = g(z^{-1}) (\!(u,v)\!),
\end{equation}
where $f$, $g$ are Laurent polynomials in $z_{i,\mu}$, and $g(z^{-1})$
means that we replace all variables $z_{i,\mu}$ by $z_{i,\mu}^{-1}$.
        
\begin{Lemma}\label{lem:laur-polyn-valu}
    $(\!(\ ,\ )\!)$ takes values in $\bigotimes_{i\in I_0}
    \QQ(q_s)[z_{i,\mu}^\pm]_{\mu=1,\dots,m_i}^{S_{m_i}}$ on
    $V(\lambda)$, where $S_{m_i}$ is the symmetric group permuting
    $z_{i,\mu}$ ($\mu=1,\dots, m_i$).
\end{Lemma}

\begin{proof}
    It is enough to show that $(G(b), G(b'))$ is symmetric for $b$,
    $b'\in\mathscr B(\lambda)$. The assertion is clear from
    \cite[Th.~4.16]{MR2066942}.
\end{proof}

\begin{Remark}
    When $\mathfrak g$ is symmetric, the bilinear form $(\!(\ ,\ )\!)$
    is coming from the intersection pairing on the equivariant
    $K$-theory of quiver varieties by \cite{VV-can}. Therefore the
    existence of $(\!(\ ,\ )\!)$ and \lemref{lem:laur-polyn-valu} are
    apparent in this context.
\end{Remark}

\subsection{Affine cellularity}

Let $\widetilde\bU$ be the modified quantum affine algebra of level
$0$. For a $\lambda$ as above, let $\widetilde\bU[{}^>\lambda]$ be
the two sided ideal consisting of all elements $x\in\widetilde\bU$
acting on $V(\lambda')$ by $0$ for any $\lambda'\not>\lambda$. We define $\widetilde\bU[{}^\ge\lambda]$ in the same way. We thus have a chain of two sided ideals in $\widetilde\bU$ for various $\lambda$'s.

In \cite[\S6]{MR2066942} we showed that
$\widetilde\bU[{}^\ge\lambda]$, $\widetilde\bU[{}^>\lambda]$ are
compatible with the global crystal base of $\widetilde\bU$, and give a
description of the induced base of $\widetilde\bU[\lambda] =
\widetilde\bU[{}^\ge\lambda]/\widetilde\bU[{}^>\lambda]$. In
particular, it was shown that the induced base (or the underlying {\it
  abstract\/} crystal) is parameterized by
\begin{equation}\label{eq:4}
    \mathscr B_W(\lambda)\times\operatorname{Irr}G_\lambda\times
    \mathscr B_W(\lambda),
\end{equation}
where $\mathscr B_W(\lambda)$ is a certain finite set, $G_\lambda =
\prod_i GL(m_i)$, and $\operatorname{Irr}G_\lambda$ is the set of
irreducible representations of $G_\lambda$.
The corresponding global base elements are of the form
\begin{equation*}
    G_\lambda(b,s,b') = G(b) S G(b')^\# \bmod \bU[{}^>\lambda],
\end{equation*}
where $G(b)$ (resp.\ $S$) is the global base element of
$\widetilde\bU$ corresponding to $b\in \mathscr B_W(\lambda)$ (resp.\
$s\in\operatorname{Irr}G_\lambda$), and ${}^{\#}$ is a certain
anti-involution of $\widetilde\bU$.

Moreover the product $\mathscr B_W(\lambda)\times\operatorname{Irr}
G_\lambda$ of the first and second factor in \eqref{eq:4} is
identified with the underlying set $\mathscr B(\lambda)$ of the global
base of $V(\lambda)$ by
\begin{equation*}
    (b,s) \mapsto G(b)S u_\lambda\in V(\lambda).
\end{equation*}
Similarly the product of the second and third factor gives also the
global base of $V(\lambda)$ by $(s,b') \mapsto (S G(b')^\#)^\# =
G(b')^\# S^\#$, and $S^\#$ corresponds to the dual representation of
$s$.

Multiplication of two global base elements are expressed by the
bilinear form $(\ ,\ )$ on $V(\lambda)$ by
\begin{equation}\label{eq:1}
    \begin{split}
    & G_\lambda(b_1,s_1,b_1') G_\lambda(b_2,s_2,b_2')
\\
    =\; & q^n
    \sum_{s''\in\operatorname{Irr} G_\lambda}
    (G(b_2)S_2 u_\lambda, G(b_1') S'' u_\lambda)
    G(b_1) S_1 S'' G(b_2')^\# \bmod \Tilde{\mathbf U}[{}^>\lambda],
    \end{split}
\end{equation}
where $n= (\wt b_1', 2\lambda + \wt b_1')/2$. See
\cite[Lemma~6.17]{MR2066942}. Our goal is to show that this
immediately implies the affine cellularity thanks to a reformulation
of $(\ ,\ )$ in the previous subsection.

From \eqref{eq:2}, we have
\begin{equation*}
    \begin{split}
    &(G(b_2)S_2 u_\lambda, G(b_1') S'' u_\lambda)
\\
    = \; & \left[(\!( G(b_2) S_2 u_\lambda, G(b_1') S'' u_\lambda)\!)
         \prod_{i} \frac1{m_i!} 
      \prod_{\nu\neq\mu}\left( 1 - z_{i,\mu} z_{i,\nu}^{-1}\right)
       \right]_1
\\
    = \; & \left[(\!( G(b_2) u_\lambda, G(b_1') u_\lambda)\!)
         s_2(z) s''(z^{-1}) 
         \prod_{i} \frac1{m_i!} 
      \prod_{\nu\neq\mu}\left( 1 - z_{i,\mu} z_{i,\nu}^{-1}\right)
       \right]_1,
    \end{split}
\end{equation*}
where we have used \eqref{eq:3} in the second equality.

By \cite[Chap.~VI, \S9]{Mac}, 
\begin{equation*}
    \left[ f(z) g(z^{-1}) \prod_{\mu\neq\nu} \left( 1 - z_\mu z_\nu^{-1}\right)
      \right]_1
\end{equation*}
is the standard inner product on the symmetric polynomials $f$, $g$ of 
$m$-variables $z = (z_1,\dots,z_m)$.
Since Schur functions gives an orthonormal base, we have
\begin{equation*}
    \sum_{s''\in\operatorname{Irr} G_\lambda}
    (G(b_2)S_2 u_\lambda, G(b_1') S'' u_\lambda) s''(z)
    = (\!(G(b_2) u_\lambda, G(b_1') u_\lambda)\!) s_2(z).
\end{equation*}

Therefore the right hand side of \eqref{eq:1} is
\begin{equation*}
    q^n
    G(b_1) S_1 S_2 (\!(G(b_2) u_\lambda, G(b_1') u_\lambda)\!)
    G(b_2')^\# \bmod \widetilde\bU[{}^>\lambda].
\end{equation*}
This equality means that $\widetilde\bU[\lambda]$ is a generalized
algebra over $R(G_\lambda)$, the representation ring of $G_\lambda$,
where the bilinear form $\psi$ (appeared in \cite[Prop.~2.2]{KX}) is
$(\!(\bullet u_\lambda,\bullet u_\lambda)\!)$. The other ingredients,
the anti-involution $i$ on $\widetilde\bU$ is ${}^{\#}$, and $\sigma$
is the induced involution on $R(G_\lambda)$, given by the dual
representation.

Therefore $\widetilde\bU$ satisfies the axioms of affine cellular
algebras from \cite{KX} except that the chain of two-sided ideals has
the infinite lengths. If we want to cut out to a finite chain, we
just need to consider quotients of $\widetilde\bU$, called {\it BLN
  algebras\/} as in \cite{BLN}.

\subsection{Idempotents}

By \cite[Th.~3.12]{KX} the affine cellularity of $\widetilde\bU$ gives
us a classification of its simple modules. More precisely, isomorphism
classes of simple modules of $\widetilde\bU$ are parametrized by the
{\it open\/} subset of the set of maximal ideals $\mathfrak
m\in\operatorname{MaxSpec} R(G_\lambda)$ such that $(\!(\
,\ )\!)$ is not identically zero on $R(G_\lambda)/\mathfrak m$.

On the other hand, a classification of simple modules of
$\widetilde\bU$ is well-known: it is the same as those of the usual
quantum affine algebra $\bU$, and is given by Drinfeld polynomials. It
means that simple modules correspond to the whole
$\operatorname{MaxSpec} R(G_\lambda)$, {\it not\/} its proper open
subset. We directly check this assertion in this section.

It is clear that that the key is the value of $(\!(\ ,\ )\!)$ at the
extremal vector $u_\lambda$, as the Drinfeld polynomial is given by
eigenvalues of a commuting family of elements in $\bU$.

One of the defining property of $(\ ,\ )$ is
\begin{equation*}
    (u_\lambda,u_\lambda) = 1.
\end{equation*}
From the definition of $(\!(\ ,\ )\!)$, we have $(\!(u_\lambda,
u_\lambda)\!) = 1$. Therefore the pairing $(\!(\ ,\ )\!)$ is never
zero, hence the condition $(\!(\ ,\ )\!)$ is nonzero on
$R(G_\lambda)/\mathfrak m$ is vacuous.

Thanks to \cite[Th.~4.1(1)]{KX}, this condition is equivalent to that
all cell ideals $\widetilde\bU[{}^>\lambda]$ are idempotent.

There is the distinguished element in $\mathscr B(\lambda)$,
corresponding to $u_\lambda$. (We may assume that the
$\operatorname{Irr} G_\lambda$-component is the trivial representation
$1$ in the description $\mathscr B(\lambda) = \mathscr
B_W(\lambda)\times\operatorname{Irr}G_\lambda$.)
As the global base element in $\widetilde\bU$, it is the projector
$a_\lambda$ to the weight $\lambda$-space. In particular, it is an
idempotent.

\begin{Remark}
In the geometric picture, $a_\lambda$ is the class of the diagonal
$\Delta\M(\lambda,\lambda)\subset Z(\lambda)$, where
$\M(\lambda,\lambda)$ is a distinguished component of $\M(\lambda)$,
consisting of a single point.
\end{Remark}

Therefore two conditions required in \cite[Th.~4.4]{KX} are satisfied
for $\widetilde\bU$, or more precisely for its quotients, BLN
algebras. Hence
\begin{Theorem}
    A BLN algebra is of finite global dimension, and its derived
    category admits a stratification whose sections are equivalent to
    derived categories of $R(G_\lambda)$.
\end{Theorem}

\subsection{Approach via quiver varieties}

Suppose again that $\g$ is symmetric. We follow the notation in
\cite{MR2058969}. Let ${}_{\mathcal A}\widetilde\bU$ be the
$\ZZ[q,q^{-1}]$-form of $\widetilde\bU$.
Let $Z(\lambda)$ denote the analog of the Steinberg variety for quiver
varieties, which is the fiber product
$\M(\lambda)\times_{\M_0(\lambda)} \M(\lambda)$. The algebra
homomorphism $\Phi_\lambda\colon {}_{\mathcal A}\widetilde\bU\to
K^{\CC^*\times G_\lambda}(Z(\lambda))$, constructed in \cite{Na-qaff}
factors through ${}_{\mathcal A}\widetilde\bU/ {}_{\mathcal
  A}\widetilde\bU[{}^>\lambda]$.
(The notation $\Phi_\lambda$ has been used already above, but it should be clear from the context.)
Remark that it has been shown that $\Phi$ is an algebra homomorphism
to $K^{\CC^*\times G_\lambda}(Z(\lambda))/\mathrm{torsion}$ in
\cite{Na-qaff}. The proof is a reduction to the case
$\g_0=\algsl_2$. The reduction argument works without dividing by
torsion. Therefore it is enough to check the relation holds for
$\g_0=\algsl_2$.
In this case, the corresponding quiver varieties are
cotangent bundles to Grassmannians (of various dimensions). It is
known that $K^{\CC^*\times G_\lambda}(Z(\lambda))$ is free
\cite[Cor.~6.2.6]{CG}. Therefore it is unnecessary to divide by torsion.

Let us consider the open embedding $j\colon Z^\circ(\lambda) \to
Z(\lambda)$, the inverse image of $\M_0(\lambda)\setminus \{0\}$ in
the fiber product. The complement $Z(\lambda)\setminus
Z^\circ(\lambda)$ is $\La(\lambda)\times\La(\lambda)$, where
$\La(\lambda)$ is the lagrangian subvariety in $\M(\lambda)$, the
inverse image of $0$ under $\M(\lambda)\to \M_0(\lambda)$. Let
$i\colon \La(\lambda)\times\La(\lambda)\to Z(\lambda)$ be the closed
immersion.
We have pull-back $j^*$ and push-forward $i_*$ homomorphisms, which
fits in the commutative diagram such that both horizontal sequences
are exact:
    \begin{equation}\label{eq:5}
        \begin{CD}
      @.
        K^{\CC^*\times G_\lambda}(\La(\lambda)\times\La(\lambda))
        @>{i_*}>>
        K^{\CC^*\times G_\lambda}(Z(\lambda))
        @>{j^*}>> K^{\CC^*\times G_\lambda}(Z^\circ(\lambda)) @>>> 0
\\
      @. @A{\spadesuit}AA @AA{\Phi_\lambda}A @AAA @.
\\
      0 @>>> 
      {}_{\mathcal A}\widetilde\bU[{\lambda}] = 
      {}_{\mathcal A}\widetilde\bU[{}^\ge\lambda]
      /{}_{\mathcal A}\widetilde\bU[{}^{> \lambda}]
      @>>>
      {}_{\mathcal A}\widetilde\bU/{}_{\mathcal A}\widetilde\bU[{}^{> \lambda}] @>>>
      {}_{\mathcal A}\widetilde\bU/{}_{\mathcal A}\widetilde\bU[{}^\ge\lambda] @>>> 0.
        \end{CD}
    \end{equation}

    The statement that the restriction of $\Phi_\lambda$ to ${}_{\mathcal
      A}\widetilde\bU[\lambda]$ is proved as follows: As we have
    explained above, ${}_{\mathcal A}\widetilde\bU[\lambda]$ has a
    base consisting of elements $G_\lambda(b,s,b') = G(b) S
    G(b')^\#\bmod\bU[{}^>\lambda]$. The element $S$ is identified
    with an irreducible representation of $G_\lambda$, and is sent to
    the class $R(\CC^*\times G_\lambda) = K^{\CC^*\times
      G_\lambda}(\La(\lambda,\lambda)\times\La(\lambda,\lambda))$,
    where $\La(\lambda,\lambda) = \M(\lambda,\lambda)$ is the
    distinguished component of $\M(\lambda)$, consisting of a single
    point, mentioned above. Then from the definition of the
    convolution product, $K^{\CC^*\times
      G_\lambda}(\La(\lambda)\times\La(\lambda))$ is a bimodule, and
    hence the assertion follows. And it also follows that
    $j^*\circ\Phi$ factors through ${}_{\mathcal
      A}\widetilde\bU/{}_{\mathcal A}\widetilde\bU[{}^\ge\lambda]$.

    Moreover, K\"unneth formula holds for $\La(\lambda)$
    \cite[Th.~3.4]{Na-tensor}, hence 
\[  
K^{\CC^*\times G_\lambda}(\La(\lambda)\times\La(\lambda))\cong
    K^{\CC^*\times G_\lambda}(\La(\lambda))\otimes_{R(\CC^*\times G_\lambda)}
    K^{\CC^*\times G_\lambda}(\La(\lambda)).
\]
Since $K^{\CC^*\times G_\lambda}(\La(\lambda))$ is isomorphic to the
extremal weight module $V(\lambda)$, it follows that $\spadesuit$ is
an isomorphism.
Therefore ${}_{\mathcal A}\widetilde\bU[{\lambda}]$ has a structure of
a generalized matrix algebra, and hence we see why $\widetilde\bU$ is
affine cellular.

Let us turn to the convolution algebra $K^{\CC^*\times
  G_\lambda}(Z(\lambda))$.
We have a stratification $\M_0(\lambda) =
\bigsqcup\Mreg_0(\mu,\lambda)$, where $\mu$ runs the set of (level 0)
dominant weights with $\mu\le\lambda$. The closure order is the
opposite of the dominance order.
Let $Z(\lambda) = \bigsqcup Z(\lambda)_\mu$ be the corresponding
decomposition of $Z(\lambda)$, and let $Z(\lambda)_{\ge\mu} =
\bigsqcup_{\mu'} Z(\lambda)_{\mu'\ge\mu}$, $Z(\lambda)_{>\mu} =
\bigsqcup_{\mu'} Z(\lambda)_{\mu'>\mu}$. Both are closed subvarieties
in $Z(\lambda)$. And $Z(\lambda)_\mu = Z(\lambda)_{\ge\mu}\setminus Z(\lambda)_{>\mu}$ is open in $Z(\lambda)_{\ge\mu}$.
We consider a variant of \eqref{eq:5}:
\begin{equation*}
    \begin{CD}
 @. K^{\CC^*\times G_\lambda}(Z(\lambda)_{>\mu}) @>{i_*}>>
 K^{\CC^*\times G_\lambda}(Z(\lambda)_{\ge\mu}) @>>>
 K^{\CC^*\times G_\lambda}(Z(\lambda)_{\mu}) @>>> 0
\\
      @. @AAA @AAA @AAA @.
\\
      0 @>>> 
      {}_{\mathcal A}\widetilde\bU[{}^>\mu]/
      {}_{\mathcal A}\widetilde\bU[{}^>\lambda]
      @>>>
      {}_{\mathcal A}\widetilde\bU[{}^\ge\mu]/
      {}_{\mathcal A}\widetilde\bU[{}^>\lambda]
      @>>>
      {}_{\mathcal A}\widetilde\bU[\mu] @>>> 0.
    \end{CD}
\end{equation*}

\begin{Lemma}
    $i_*$ becomes injective, if we divide domain and target by
    $R(\CC^*)$-torsion.
\end{Lemma}

\begin{proof}
    \begin{NB}
    By \cite[Th.~6.1.22]{CG}, $K^{\CC^*\times G_\lambda}(\ )$ is the
    Weyl group invariant part of $K^{\CC^*\times T_\lambda}(\ )$,
    where $T_\lambda$ is a maximal torus of $G_\lambda$. Therefore it
    is enough to show the claim for $K^{\CC^*\times T_\lambda}(\
    )$. 
    \end{NB}%
    If we consider the fixed point $Z(\lambda)^{\CC^*}$, it is
    contained in $Z(\lambda)_\lambda =
    (\La(\lambda)\times\La(\lambda))^{\CC^*}$. It is because the
    action of $\CC^*$ on an affine space containing $\M_0(\lambda)$
    has only positive weights, and hence $0$ is the only fixed point
    in $\M_0(\lambda)$. Since the projective morphism
    $\pi\colon\M(\lambda)\to \M_0(\lambda)$ is $\CC^*$-equivariant,
    the assertion follows.

    Now the localization theorem in the equivariant $K$-theory implies
    that $i_*$ becomes an isomorphism if we localize the equivariant
    $K$-group at $\mathrm{Frac}(R(\CC^*))$, the fractional field of
    $R(\CC^*) = \ZZ[q,q^{-1}]$. We are done, as the kernel of $\bullet\to
    \bullet\otimes \mathrm{Frac}(R(\CC^*))$ is the torison part.
    \begin{NB}
    (In fact, $K^{\CC^*\times T_\lambda}(\La(\lambda))$ is known to be
    free \cite[\S7]{Na-qaff}.)
    \end{NB}%
\end{proof}

We ignore the torsion part hereafter.

The homomorphism $K^{\CC^*\times G_\lambda}(Z(\lambda)_{\ge\mu})\to
K^{\CC^*\times G_\lambda}(Z(\lambda))$, which is injective by above,
is compatible with the convolution product. Therefore $K^{\CC^*\times
  G_\lambda}(Z(\lambda)_{\ge\mu})$ is a two-sided ideal. The same holds for $K^{\CC^*\times G_\lambda}(Z(\lambda)_{>\mu})$.
Therefore we need to analyze $K^{\CC^*\times
  G_\lambda}(Z(\lambda)_{\mu})$ to show that $K^{\CC^*\times
  G_\lambda}(Z(\lambda))$ is affine cellular.

For the original Steinberg variety of type $A$, we consider a short
exact sequence
\begin{equation*}
    0\to K^{\CC^*\times G}(Z_{\overline{O}\setminus O})
    \to K^{\CC^*\times G}(Z_{\overline{O}})
    \to K^{\CC^*\times G}(Z_{{O}}) \to 0,
\end{equation*}
where $O$ is a nilpotent orbit and $\overline{O}$ is its closure.
(See \cite{MR2235345} for the relevance of the above short exact
sequence for the structure of cells of affine Hecke algebras of type
$A$.)
In this case, $Z_O$ is a fiber bundle over $O$ whose fiber at $e$ is
$\mathcal B_e\times\mathcal B_e$, where $\mathcal B_e$ is the Springer
fiber at $e$.
In the quiver variety case, $Z(\lambda)_\mu$ is a fiber bundle whose
fibers are isomorphic to $\La(\mu)\times\La(\mu)$.
However the base $\M^{\operatorname{reg}}_0(\mu,\lambda)$ is not an
orbit of $G_\lambda$, and hence we need a further study.

If we replace equivariant $K$-group by equivariant homology groups, we
still have a similar diagram, where the bottom row is replaced by
Yangian. But we still need to analyze $H^{\CC^*\times
  G_\lambda}_*(Z(\lambda)_\mu)$.
  S.~Kato has studies this problem in his study of extension algebra
  \cite{Kato-ext}. In his case, $Z^\circ(\lambda)$ is replaced by an
  union of orbits, and hence the picture is similar to the case of the original Steinberg variety.
We will come back to this problem in near future.

It is also clear from our perspective that we need to care extra
stratum, not of a form $\Mreg_0(\mu,\lambda)$ in $\M_0(\lambda)$ for
quiver varieties of infinite types. For example, symmetric powers of
$\CC^2/\Gamma$ appear for affine types. Here $\Gamma$ is a finite
subgroup of $SU(2)$.

\begin{NB}
Let us give a direct prove that $\spadesuit$ is an isomorphism, based
on well-known properties on equivariant $K$-theory of quiver
varieties. The proof seems to be applicable to much more general situation, say convolution algebras for graded quiver varieties.

\begin{Proposition}
    $\spadesuit$ is an isomorphism.
\end{Proposition}

\begin{proof}
    Suppose that $x\in{}_{\mathcal A}\widetilde\bU[{}^\ge\lambda]$
    is sent to $0$ under $\spadesuit$. Since $K^{\CC^*\times
      G_\lambda}(\La(\lambda))$ is isomorphic to ${}_{\mathcal
      A}V(\lambda)$ (\cite[Th.~2]{MR2074599}), it means that $x$ acts
    $0$ on $V(\lambda)$. From the definition, we have
    $x\in{}_{\mathcal A}\widetilde\bU[{}^>\lambda]$. Therefore $x=0$
    in ${}_{\mathcal A}\bU[\lambda]$. Thus $\spadesuit$ is injective.

    It is known that $K^{\CC^*\times G_\lambda}(\La(\lambda))$, and
    hence $V(\lambda)$, is cyclic, that is ${}_{\mathcal A}\widetilde
    \bU\cdot u_\lambda = {}_{\mathcal A}V(\lambda)$. (This is also a
    direct consequence of the definition of the extremal weight module
    ${}_{\mathcal A}V(\lambda)$.) Now it follows that $({}_{\mathcal
      A}\widetilde \bU) a_\lambda ({}_{\mathcal A}\widetilde \bU)$
    surjects to 
\(    
    K^{\CC^*\times G_\lambda}(\La(\lambda))\otimes_{R(\CC^*\times G_\lambda)}
    K^{\CC^*\times G_\lambda}(\La(\lambda))
\)
under $\spadesuit$.
\end{proof}

Combining this proposition with the above diagram, we now see that the
affine cellular structure on ${}_{\mathcal A}\widetilde\bU$ can be
proved without using \cite{MR2066942} for symmetric $\g$.
\end{NB}

\bibliographystyle{myamsplain}
\bibliography{nakajima,mybib,note}

\end{document}